\documentclass[11pt]{amsart}
\usepackage[top=3.5cm,bottom=3.5cm,left=3.8cm,right=3.8cm]{geometry}               
\usepackage{amsmath,amscd,amssymb,amsthm,stmaryrd}
\usepackage{mathtools}
\usepackage{enumerate}
\usepackage{setspace}
\usepackage{mathrsfs}
\usepackage{color}
\usepackage[numbers]{natbib}
\usepackage{hyperref}
\usepackage{tikz-cd}

\usepackage{cleveref}

\DeclareMathOperator{\gal}{Gal}

\DeclareMathOperator{\disc}{disc}

\def\F{\mathbb{F}}

\def\Z{\mathbb{Z}}

\def\Q{\mathbb{Q}}

\def\O{\mathcal{O}}
\def\U{\mathcal{U}}

\def\calV{\mathcal{V}}
\def\calO{\mathcal{O}}

\def\fgen{\tilde f}

\theoremstyle{definition}
\newtheorem{definition}{Definition}[section]

\theoremstyle{plain}
\newtheorem{theorem}[definition]{Theorem}
\newtheorem{corollary}[definition]{Corollary}
\newtheorem{lemma}[definition]{Lemma}

\author[A. Ferraguti]{Andrea Ferraguti}
\address{Dipartimento di Matematica, Università di Torino, 10123 via Carlo Alberto 10, Torino, Italy}
\email{andrea.ferraguti@unito.it}
\author[G.M. Lido]{Guido Maria Lido}
\address{Dipartimento di Matematica, Università di Roma  Tor Vergata, Via della Ricerca Scientifica, 1}
\email{guidomaria.lido@gmail.com}

\title[]{On the dynamical Galois group of certain affine polynomials in positive characteristic}

\thanks{
Both authors have been partially supported by the ``National Group for Algebraic and Geometric Structures, and their Applications" (GNSAGA - INdAM).
\\
The second author is supported by the MIUR Excellence Department Project MatMod@TOV awarded to the Department of Mathematics, University of Rome Tor Vergata, by the PRIN PNRR 2022 ``Mathematical Primitives for Post Quantum Digital Signatures''.
\\
The authors are grateful to the Mathematics Departments of Torino and Roma Tor Vergata.
}
\subjclass[2010]{}


\begin{document}

\begin{abstract}
	We use explicit class field theory of rational function fields to prove a dynamical criterion for a polynomial of the form $x^{p^r}+ax+b$ over a field of characteristic $p$ to have dynamical Galois group as large as possible. When $p=2$ and $r=1$ this yields an analogue in characteristic $2$ of the celebrated criterion of Stoll for quadratic polynomials over fields of characteristic not $2$.
\end{abstract}
\maketitle
\section{Introduction}

Let $K$ be a field, let $f\in K[x]$ be a polynomial of degree $d>1$ and let $\alpha\in K$. The inverse limit $G_\infty(f,\alpha)$ of the Galois groups of $f^n-\alpha$ is nowadays known as the \emph{dynamical Galois group} of the pair $(f,\alpha)$; when $\alpha=0$ one simply speak about the dynamical Galois group of $f$. If $f^n-\alpha$ is separable for every $n$ such group embeds in the automorphism group of the infinite, regular, rooted $d$-ary tree. The question of determining how large $G_\infty(f,\alpha)$ is with respect to the full automorphism group of the tree has been first investigated by Odoni \cite{odoni}, and subsequently by various other authors (see \cite{jones2} for a survey on the topic).

One of the very first instances of this problem was posed by Cremona in 1989 in \cite{cremona}: the author conjectured therein that the dynamical Galois group of $x^2+1\in \Q[x]$ is as large as possible, namely it is the full automorphism group of the relevant tree. This was proved three years later by Stoll in \cite{stoll}; the first step in the paper was the proof of a criterion for a quadratic polynomial to have dynamical Galois group as large as possible. Although the criterion was only stated over $\Q$, it is easy to show that the same holds true over any field of characteristic not $2$. We recall it below; we remark that it is also easy to adapt it to any basepoint $\alpha$. For a polynomial $f=x^2+c$, the \emph{adjusted post-critical orbit} is the sequence defined by $c_1=-c$ and $c_n=f(c_{n-1})$ for every $n\ge 2$.
\begin{theorem}[\cite{stoll}]\label{stoll_theorem}
	Let $K$ be a field of characteristic not $2$ and let $f=x^2+c\in K[x]$ be such that $c_n\ne 0$ for every $n\ge 1$. Let $\Omega_\infty$ be the full automorphism group of the infinite, rooted, regular, binary tree. Then $G_\infty(f)=\Omega_\infty$ if and only if the adjusted post-critical orbit $\{c_1,c_2,\ldots\}$	is a linearly independent set in the $\F_2$-vector space $K^\times/{K^{\times}}^2$.
\end{theorem}
The fact that $f$ is monic and missing the linear term is not restrictive, since up to conjugacy we can always assume it holds (see for example \cite{ferraguti} for a definition of conjugacy)). The condition $c_n\ne 0$ is also not restrictive, since it is necessary to have $G_\infty(f)=\Omega_\infty$.

Theorem \ref{stoll_theorem} encompasses therefore the totality of quadratic polynomials; on the other hand clearly its statement does directly generalize to fields of characteristic $2$. In fact, over such fields polynomials of the form $x^2+c$ are not separable, and if we keep the linear term and look at $x^2+ax+b$ there are no finite critical points, and therefore the post-critical orbit is not even defined.
This rises the question of what happens to quadratic polynomials in characteristic $2$.

The present paper stemmed from the above question, and the answer we found holds for the more general class of polynomials of the form $x^q+ax+b$ in characteristic $p$, where $q$ is a $p$-power. In order to describe our result, first we prove that for polynomials of this form, the dynamical Galois group can never be $\Omega_\infty$, and it is always contained in a smaller subgroup. More precisely, we show that if $K$ is an $\F_q$-field, namely an extension of $\F_q$, with $t,s$ algebraically independent variables over $K$ and $f_{t,x}=x^q+tx+s$, then $G_\infty(f_{t,s})\cong \Phi_\infty^{(q)}$, where
$$
\Phi_{\infty}^{(q)} \cong \F_q[\![ y ]\!] \rtimes \F_q[\![ y ]\!] ^\times,
$$
see Lemma \ref{gal_gp_structure}. This was already proven by Juul in \cite{juul} for $q=2$. We remark that $\Phi_\infty^{(q)}$ has a normal pro-$p$-subgroup of index $q-1$:
$$
\Phi_{\infty, 1}^{(q)} \cong \F_q[\![ y ]\!] \rtimes (1+y\F_q[\![ y ]\!]).
$$
Throughout the whole paper, we will fix a prime $p$ and a power $q=p^r$ that will never change, and we will therefore drop the $(q)$ superscript to ease the notation.

Now, specializing $t,s$ to $t_0,s_0\in K$ with $t_0\ne 0$ we get a polynomial $f_{t_0,s_0}=x^q+t_0x+s_0\in K[x]$ and an embedding 
\begin{equation} \label{eq:def_iota}
\iota  \colon  G_\infty(f_{t_0,s_0}) \hookrightarrow \Phi_{\infty},
\end{equation}
defined in \eqref{eq:pedantic_iota}, that is canonical up to conjugation in $\Phi_\infty$.
Our main theorem is then the following.

\begin{theorem}\label{thm:main}
    Let $K$ be an $\F_q$-field, let $t_0, s_0\in K$ with $t_0\neq 0$, and let $f = f_{t_0,s_0}  = x^q+t_0 x+s_0 \in K[x]$. Fix an $\F_p$-basis $(\gamma_1,\ldots,\gamma_r)$ of $\F_q$. If $q>2$, then the map $\iota$ given in \eqref{eq:def_iota} is an isomorphism if and only if the following three conditions hold:
    \begin{enumerate}
        \item\label{item:1} $-t_0$ is not an $\ell$-th power in $K$ for all primes $\ell$ dividing $p-1$;
        \item\label{item:3} the set
        $$\left\{\frac{\gamma_j}{t_0^{i}}\colon j\in \{1,\ldots,r\}, \, i\ge 1,\,  p\nmid i\right\}$$
        is $\F_p$-linearly independent in $\calV(K) = K/\{z^p-z : z\in K\}$;
        \item\label{item:2} $f$ is irreducible.
    \end{enumerate}
    
    Moreover, for all $q$, if there exists $\theta_0\in K$ with $\theta_0^{q-1} =-t_0$, then the image of $\iota$ is contained in $\Phi_{\infty, 1}$. In this case $\iota$ is an isomorphism with $\Phi_{\infty, 1}$ if and only if the set
    $$\left\{\frac{	\gamma_js_0}{\theta_0^q}\colon j\in \{1,\ldots,r\}\right\}\cup\left\{\frac{\gamma_j}{t_0^i}\colon j\in \{1,\ldots,r\},\,i\ge 1,\,p\nmid i\right\}$$
    is $\F_p$-linearly independent in $\calV(K)$.
\end{theorem}
A couple of remarks are in order here. First of all, Theorem \ref{thm:main} takes care of any basepoint $\alpha\in K$ as well. In fact, the pair $(x^q+t_0x+s_0,\alpha)$ is $K$-conjugate to the pair $((x+\alpha)^q+t_0(x+\alpha)+s_0-\alpha,0)=(x^q+t_0x+\alpha^q+t_0\alpha+s_0-\alpha,0)$. For $q=2$, Theorem \ref{thm:main} takes also care of any quadratic polynomial, not necessarily monic, by conjugating by $ux$ for an appropriate $u\in K$.

Next, for $q=2$ the groups $\Phi_\infty$ and $\Phi_{\infty,1}$ coincide, and of course every element of $K$ is a $(q-1)$-th power. Hence, in this case the second part of the statement yields a surjectivity criterion for $\iota\colon G_\infty\to \Phi_\infty$, completely analogous to Stoll's Theorem \ref{stoll_theorem}.
When $q>2$, asking that $-t_0$ is a $(q-1)$-th theorem in $K$ can be seen as an analogue of asking that the base field contains an appropriate root of unity in characteristic $0$. This analogy is relevant when considering the generalization of Stoll's Theorem to unicritical polynomials. For example, when $g=x^p+c$ with $p$ prime over a field $K$ of characteristic not $p$ containing a primitive $p$-th root of unity, then the criterion for surjectivity (in the appropriate group) involves the linear independence of the post-critical orbit in $K^\times/{K^{\times}}^p$.

Finally, one can wonder if and how Theorem \ref{thm:main} relates to the dynamics of $f_{t_0,s_0}$, since Theorem \ref{stoll_theorem} clearly does. Here one can notice the following: it is easy to show by induction that if $f=x^q+t_0x+s_0$ and $f^n=\sum_{0=1}^{q^n}a_ix^i$, then $a_1=t^n$ and every other monomial appearing in $f^n$ has degree a power of $p$. Hence, the derivative of $f^n$ is $t_0^n$. Since the discriminant of $f^n$ is nothing but the resultant of $f^n$ and its derivative, and the latter is constant, we see easily that $\disc(f^n)=t_0^{nq^n}$. Now if $n=p^mk$ with $p\nmid k$, it is an elementary observation that $1/t_0^{nq^n}$ and $1/t_0^k$ coincide in $\calV(K)$. Hence point $(2)$ of Theorem \ref{thm:main} can be restated by asking that the set 
$$
\left\{\frac{\gamma_j}{\disc(f^n)}\colon j \in\{1,\ldots,r\},\,\, n\ge 1, p\nmid n\right\} \ ,
$$
spans in $\calV(K)$ the largest possible subspace. Such set is a dynamical object associated with $f$, and its arithmetic determines the dynamical Galois group of $f$, exactly as in Stoll's Theorem \ref{stoll_theorem}. There is  an additional consideration to be made here. It is a very 
well-known in arithmetic dynamics (see for example \cite[Lemma 2.6]{jones1}) that if $g=x^p+c$ is a unicritical polynomial over a field of characteristic $0$, then $\disc(g^n)$ coincides, up to $p$-th powers, with $f^n(0)$. Hence in this setting studying the post-critical orbit modulo $p$-th powers is the same thing as studying the sequence of the discriminants of the iterates modulo $p$-th powers. In the setting of our Theorem \ref{thm:main}, the post-critical orbit cannot be defined, but the sequence of the discriminants is once again the relevant object.

It is easy to use Theorem \ref{thm:main} to find examples of polynomials for which $\iota$ is an isomorphism. For example, if $f=x^2+tx+t^3$ over $K = \F_2(t)$ or $K = \overline{\F}_2(t)$ then $G_\infty(f)=\Phi_\infty=\Phi_{\infty,1}$, since there is no rational function $u\neq 0,1$ such that $u^2-u = a_1 t + a_{-1}t^{-1} + a_{-3}t^{-3} + \ldots a_{2d+1} d^{-2d-1}$ with $a_i\in \F_2$: indeed $u^2-u$ can only have poles of even orders, forcing $u^2-u = a_1 t$ and in turn $a_1 = 0$. By similar arguments the same conclusion holds for $K= \overline{\F}_2(C)$ with $C$ a curve over $\overline{\F}_2$ and $f(x) = x^2+ax+b$, and $a,b$ are rational functions on $C$ such $a$ has a zero $P$ of odd order which is not a pole for $b$, and $b$ has a pole $Q$ of odd order which is not a pole nor a zero for $a$.

The key idea we employ to prove Theorem \ref{thm:main} is to use Hayes' explicit description of the class field theory of $\F_p(t)$ using torsion points for the Carlitz module, that can be found in \cite{hayes}.

\section*{Notation}
Throughout the paper, we fix once and for all a prime $p$, a power $q=p^r$ and an $\F_p$-basis $(\gamma_1,\ldots,\gamma_r)$ for $\F_q$ . For a field $K$ of characteristic $p$ we let
$$
 \mathcal U(K)\coloneqq \{z^p-z\colon z\in K\}, \quad \calV(K) \coloneqq K/\mathcal U(K).
$$
Notice that $\mathcal U(K)$ and $\calV(K)$ are  $\F_p$-vector spaces.

The letters $t$ and $s$ will always refer to algebraically independent variables over the base field, so we will not repeat it every time. For example, $\F_q(t)$ will always denote the rational function field of transcendence degree $1$ over $\F_q$.

For a field $K$, a polynomial $f\in K[x]$ and an element $\alpha\in K$, as per usual we let $f^n$ to be the $n$-fold composition of $f$, with $f^0=x$. Moreover, we denote by $f^{-n}(\alpha)$ the set $\{\gamma\in \overline{K}\colon f^n(\gamma)=\alpha\}$, we let $K_n(f,\alpha)\coloneqq K(f^{-n}(\alpha))$ and $G_n(f,\alpha)\coloneqq \gal(K_n(f,\alpha)/K)$. Finally, we let
$$K_\infty(f,\alpha)\coloneqq \varinjlim_n K_n(f,\alpha)  \quad\mbox{ and }\quad G_\infty(f,\alpha)=\gal(K_\infty(f,\alpha)/K)=\varprojlim_n G_n(f,\alpha).$$
\section{Preliminaries}\label{sec:preliminaries}

\subsection{Artin-Schreier theory}
Let $K$ be a field of characteristic $p$, and fix a separable closure $K^{\text{sep}}$. By Artin-Schreier theory, separable cyclic extensions of $K$ of degree $p$ are exactly those of the form $K(y)$ with $y^p-y = u$ for some $u\in K\setminus \U(K)$, and are cut out by the character $\chi_u\colon \gal(K^{\text{sep}}/K)\to \F_p$ such that $\chi_u(\sigma) = \sigma(y) - y$.
Moreover, if $K_1\coloneqq K(y_1),K_2\coloneqq K(y_2)$ are two such extensions with $y_1^p-y_1-u_1=0$ and $y_2^p-y_2-u_2=0$, cut out by $\chi_1,\chi_2$, respectively,  then for every $a,b\in \F_p$ we have that $a\chi_1+b\chi_2$ cuts out $K(y_3)$, where $y_3^p-y_3-au_1-bu_2=0$.
Finally, $K_1$ and $K_2$ coincide in $K^{\text{sep}}$ if and only if $u_1$ and $u_2$ coincide in $\calV(K)$.

\begin{lemma}\label{lem:cyclic_and_unramified}
    Let $L/\F_q(t)$ be a separable, cyclic $p$-extension. Then $L$ is unramified outside the place $t=0$ if and only if there exist a polynomial $U=\sum_{i=0}^n a_iT^i \in \F_q[T]$ with $a_i=0$ 
    for all positive $i$ with $p\mid i$, and an element $y\in K^{\text{sep}}$ such that $L$ is of the form $\F_q(t)(y)$ with $y^p-y-U(1/t)=0$. Moreover, in such an extension the place at infinity of $\F_q(t)$ is totally split if and only if $a_0=0$.
\end{lemma}
\begin{proof}
    First we recall the following fact, that can be found for example in \cite[Proposition 3.7.8]{stichtenoth}. 
    
    \begin{center}
        $(\diamondsuit)$ Let $u\in \F_q(t)\setminus \U(\F_q(t))$ and let $y\in \overline{\F_q(t)}$ be such that $y^p-y=u$. Let $P$ be a place of $\F_q(t)$. Then the cyclic $p$-extension $\F_q(t)(y)$ is unramified at $P$ if and only if there exists some $z\in \F_q(t)(y)$ such that $v_P(z^p-z-u)\ge 0$.
    \end{center}
    Thanks to $(\diamondsuit)$, the "if" part of the statement is obvious: since $U(1/t)$ is integral everywhere outside of the place $t=0$, it is enough to choose $z=0$ in $(\diamondsuit)$.
    
    To prove the converse, first notice that equivalently we can prove that a cyclic $p$-estension of $\F_q(t)$ that is unramified at all finite places is of the form  $\F_q(t)[y]/(y^p-y-U(t))$ with $U \in \F_q[T]$. Let $M$ be such an extension, let us prove that it takes the desired form. Let $y$ be a generator for $M$ such that $y^p-y-u=0$ for some $u\in \F_q(t)$, and write $u=N/D$ with $N,D\in \F_q[t]$ and coprime. Let $P$ be an irreducible factor of $D$ (if there are none, we are done). Since the place $v$ of $\F_q(t)$ associated to $P$ is unramified in $M$, by $(\diamondsuit)$ there exists $z \in \F_q(t)$ such that $v_P(z^p-z-u)\ge 0$. By the strong approximation theorem (see for example \cite[Theorem 1.6.5]{stichtenoth}), we can suppose that $z$ is integral at all places except $P$ and the place at infinity. Now if $\tilde u\coloneqq  u- (z^p- z)$, then $\tilde u \equiv u$ modulo $\U(\F_q(t))$ and moreover $\tilde u$ is integral at all the finite places where $u$ is integral and also at $P$. Hence by repeating this process we can replace $u$ with an element $u'$ in $\F_q[t]$ such that $M=\F_q(t)(y')$, with $y'^p-y'-u'=0$ 

    We proved that $u$ is a polynomial, i.e. that $u = a_0 + \ldots + a_nt^n$. Now it is easy to remove all the monomials of positive degree divisible by $p$: if $a_{mp}t^{mp}$ is a monomial appearing in $u$, with $m$ maximal, we just replace $u$ by $u-(z^p-z)$ with $z = a_{mp}^{1/p}t^m$.
    
    For the last part of the statement, just use the fact that places above infinity correspond to factors of $x^p-x+a_0$. This polynomial is totally split in $\F_p[x]$ if and only if $a_0=0$.
\end{proof}

\begin{corollary}\label{all_cyclic_unramified}
    The set
    $$
    S = \left\{\frac{\gamma_j}{t^i}\colon j\in \{1,\ldots,r\},i \in \Z_{>0}, p\nmid i\right\}
    $$
    is a set of independent vectors in $\calV(\F_q(t))$ that generate, by Artin-Schreier theory, all the cyclic $p$-extensions of $\F_q(t)$ that are unramified outside the place $t=0$ and totally split at the place at infinity.
\end{corollary}
\begin{proof}
    The fact that the vectors in $S$ are linearly independent in $\calV(\F_q(t))$ is an elementary verification. The rest follows by Lemma \ref{lem:cyclic_and_unramified}.
\end{proof}

\subsection{The class field theory of \texorpdfstring{$\F_q(t)$}{}}
Let $\phi\colon \F_q[t]\to \F_q(t)\{\tau\}$ be the Carlitz module, i.e.\ the Drinfeld module defined by the assignment $t\mapsto t+\tau$. Recall that for a given $M\in \F_q[t]$, the $M$-torsion of $\phi$, denoted by $\Lambda_M$, is the kernel of the $\F_q$-linear operator $\phi_M$, viewed as the linear map $\overline{\F_q(t)}\to \overline{\F_q(t)}$ obtained by $\phi(M)$, interpreting $\tau$ as the Frobenius (see also \cite[Equation 1.1]{hayes}, where $\phi_M(u)$ is denoted $u^M$). Moreover, given the Drinfeld module $\widetilde \phi \colon \F_q[t]\to \F_q(t)\{\tau\}$ defined by $t\mapsto 1/t+\tau$, we denote by $\Lambda_{1/t^i}$ the $t^i$-torsion of $\widetilde \phi$.

In \cite{hayes}, Hayes gives a description of the maximal abelian extension of $\F_q(t)$ as the composition of the constant extension $\overline{\F_q}(t)$ together with extensions generated by such torsion points, as follows.

\begin{theorem}\label{thm:Hayes}
Let $F = \F_q(t)$ and define  $\Lambda_M$ as above.
\begin{enumerate}
    \item\label{item:Lambda_Pm} If $M = P^i$ with $P \in \F_q[t]$ an irreducible polynomial, the extension $F[\Lambda_M]/F$ is abelian, totally ramified at (the place over) $P$, tamely ramified at the place at infinity with inertia degree equal to $1$ and unramified at all the other places.
    \item \label{item:compositum_finite} If $M,N \in \F_q[t]$ are coprime, the extensions $F[\Lambda_M]/F$ and $F[\Lambda_N]/F$ are linearly disjoint and have compositum $F[\Lambda_{MN}]/F$.
    \item \label{item:Lambda_1/T} For every $i\ge 1$ the extension $F[\Lambda_{t^{-i}}]/F$ is totally ramified at infinity, tamely ramified at the place $t=0$ with inertia degree equal to $1$ and ramification index equal to $q-1$, and unramified at all the other places. 
    \item \label{item:F_i} For every $i\ge 1$ the extension $F[\Lambda_{t^{-i}}]/F$ has Galois group canonically isomorphic to  
    $$
    (\F_q[t^{-1}]/t^{-i})^\times \cong \F_q^\times \times (1 + t^{-1}\F_q[t^{-1}]/t^{-i}) \ ,
    $$
    and the extension $F_i \coloneqq F[\Lambda_{t^{-i}}]^{\F_q^\times}$ of $F$  is totally ramified at the place at infinity and unramified everywhere else.
    \item \label{item:compositum_disjoint} If $M$ is a polynomial then $F[\Lambda_M]/F$ is linearly disjoint from $F_i$ for all $i\ge 1$ and their compositum is linearly disjoint from $\F_{q^k}(t)$ for all $k$.
    \item \label{item:abelian} Every finite abelian extension of $\F_q(t)$ is contained in a compositum of $\F_{q^k}(t)$, $F_i$ and $F[\Lambda_M]$ for some $i,k,M$.
\end{enumerate}
\end{theorem}
For the reader's convenience, we give precise references to \cite{hayes}.
    
    For $(1)$: abelianity is in Theorem 2.1; the part about non-ramification of most primes is in Proposition 2.2; the total ramification of $P$ is Proposition 2.2 together with 2.3 which gives the degree of the extension; the ramification of the prime at infinity is in Theorems 3.1, 3.2.

    For $(2)$: the geometric fact $\Lambda_{MN} = \Lambda_M \times \Lambda_N$ gives the compositum of the two extensions; linear disjointness comes from the computations of [$F[\Lambda_M]:F]$ in Proposition 2.3.  

    $(3)$ is symmetric to the case $P = t$ in $(1)$; the statement about the inertia degree at $0$ is analogous to Theorem 3.2.  

    $(4)$ is a consequence of $(3)$.

    $(5)$ is a finite version of Proposition 5.2

    $(6)$ is a finite version of the first part of Proposition 7.1

\begin{corollary}\label{cor:all_pexts_are_where_we_like}
    If $L$ is an abelian  extension of $F = \F_q(t)$ which is only ramified at one place $P$ and totally split at infinity, then $L$ is contained in $F[\Lambda_{P^j}]$ for some $j$.
\end{corollary}
\begin{proof}
    We refer to Theorem \ref{thm:Hayes} without mentioning. 
    By $(6)$, $L$ is contained in a compositum $\F_{q^k}(T) \cdot F[\Lambda_M] \cdot F_i$ for some $i$. We first prove that, if we take $M$ minimal with respect to divisibility, then $M = P^j$ for some $j$. Indeed, if $Q$ is another prime dividing $M$ and we write $M = Q^e N$ with $N$ coprime to $Q$, then by $(2)$ and $(5)$ the extension $F[\Lambda_{Q^e}]$ is linearly disjoint from $\F_{q^k}(T) \cdot F[\Lambda_{N}] \cdot F_i$. Since, by $(1)$, $\F_{q^k}(T) \cdot F[\Lambda_{N}] \cdot F_i$ is the maximal subextension of $\F_{q^k}(T) \cdot F[\Lambda_{M}] \cdot F_i$ not ramified at $Q$, it follows that $L\subseteq \F_{q^k}(T) \cdot F[\Lambda_{N}] \cdot F_i$, contradicting the minimality of $M$. 
    By a similar argument, since the prime at infinity is not ramified in $L$ then $L$ is contained in $\F_{q^k}(T) \cdot F[\Lambda_{P^j}]$.

    Finally, fix a place $v$ of $F[\Lambda_{P^j}]$ above infinity. By $(1)$, the residue field of $v$ is $\F_q$, and since and $\F_{q^k}(T) \cdot F[\Lambda_{P^j}]$ contains $\F_{q^k}$, then $v$ is inert in $\F_{q^k}(T) \cdot F[\Lambda_{P^j}]$, i.e.\ it has a unique extension $w$. In particular the decomposition group of  $w$ contains $\gal(\F_{q^k}(t) \cdot F[\Lambda_{P^j}]/F[\Lambda_{P^j}]) $, i.e.\ the fixed field of such a decomposition group is contained in $F[\Lambda_{P^j}]$. By abelianity the same is true for all the other places above infinity; on the other hand the fixed field of such a decomposition group is the largest subfield where the place at infinity is totally split (see e.g.\ \cite[Theorem 3.8.3]{stichtenoth}). In particular $L$ is contained in it, hence in $F[\Lambda_{P^j}]$.

\end{proof}

\section{Dynamical Galois groups of affine Carlitz polynomials}\label{sec:generic}

Let $K$ be an $\F_q$-field. Let 
$$\fgen \coloneqq x^q+tx+s  \quad \in K(t,s)[x]\ . $$ We are now going to describe the dynamical Galois group $G_\infty(\fgen,0)$. 

Such $\fgen$ is a generic (or universal) case of the polynomials of the form 
$$f = f_{t_0,s_0} = x^q+t_0x + s_0 \quad \in k[x]$$
for $k$ an extension of $K$ (hence of $\F_q$) and $t_0 \neq 0$. Notice that any such $f$ is $\F_q$-affine, since we can write 
$$f(x) = L(x) + s_0 \quad \text{ with } \quad L(x) = L_{t_0}(x) = x^q+t_0x\ , $$ 
and $L$  defines an $\F_q$-linear function on $\overline k$.
In particular the iterates of $f$ are 
\begin{equation}\label{n-th iterate}
		f^n=L^n+f^{n-1}(s_0),
	\end{equation}
and since $L^n$ is again linear, if $\alpha'$ is any root of $f^n$, then the set of roots of $f^n$ is $\{\alpha' +\beta : \beta \in \ker(L^n) \}$.

The polynomial $L^n$ is separable (since $t_0 \neq 0$, then $L$ and all its iterates have no finite critical points over $\overline k$) of degree $q^n$. Moreover we can define a canonical embedding 
\begin{equation}\label{eq:R_nqstar_emb} 
    G_n(L,0) = \gal(L^n/k)  \longrightarrow \left(\tfrac{\F_q[y]}{(y^n)} \right)^\times \,, \quad \sigma \longmapsto F_\sigma(y)
\end{equation}
defined by sending each automorphism $\sigma$ to the (unique and invertible modulo~$y^n$) polynomial $F_\sigma$ such that 
\begin{equation}\label{eq:R_nqstar_emb2} 
\sigma(\beta) = F_\sigma(L)(\beta) \quad \text{for all roots $\beta$ of $L^n$}\ .
\end{equation}
This works because the set of roots of $L^n$, which we denote $\ker(L^n)$, is naturally an $\F_q[y]$-module, with $y$ acting as $L$: if $L^n(\beta)=0$, then also $L^n(L(\beta))=0$ and more generally for all polynomials $a_0+ \ldots + a_d y^d \in \F_q[y]$, also $a_0\beta + \ldots + a_dL^d(\beta)$ is a root of $L^n$. Moreover $y^n $ acts as zero and we see that  $\ker(L^n)$ is a cyclic $\F_q[y]/(y^n)$-module: indeed if we inductively choose 
\begin{equation}\label{eq:beta_i}
\beta_1 \neq 0 \in \ker(L) \,, \quad \beta_{i+1} \in \ker(L^i) \text{ such that }L(\beta_{i+1}) = \beta_i\ ,
\end{equation}
then $\beta_n$ is a generator of $\ker(L^n)$ in the sense that every $\beta\in \ker(L^n)$ can be  written as $F(\beta_n)$ with $F$ a unique polynomial in $\F_q[y]/(y^n)$.
In particular we can define $F_\sigma$ in \eqref{eq:R_nqstar_emb} as the unique polynomial that satisfies $\sigma(\beta_n) = F_\sigma(L)(\beta_n)$ and then check that \eqref{eq:R_nqstar_emb2}  holds for all $\beta$'s; since $\sigma$ acts as a permutation, then $F_\sigma$ is invertible.

Since $L^n$ is separable and $f^n= L^n + f^{n-1}(s_0)$ is also spearable and if we choose a root $\alpha_n$ of $f^n$, together with $\beta_n$ as in \eqref{eq:beta_i}, we get a map 
\begin{equation}\label{eq:Gnf0_emb} 
    \iota_n \colon G_n(f,0) = \gal(g^n/k)  \longrightarrow \left(\tfrac{\F_q[y]}{(y^n)} \right) \rtimes \left(\tfrac{\F_q[y]}{(y^n)} \right)^\times \ , \quad \sigma \longmapsto (G_\sigma, F_\sigma) \ ,
\end{equation}
where, in the semidirect product, the left is the additive group of the ring $\tfrac{\F_q[y]}{(y^n)}$ and the group of units on the right naturally acts on it by multiplication.   
Remembering that all the roots of $f^n$ are of the form $\alpha_n + \beta$, for $\beta = F(\beta_n) \in \ker(L^n)$, we can define the map \eqref{eq:Gnf0_emb} by imposing that $G_\sigma,F_\sigma$ are the unique polynomials modulo~$y^n$ such that 
\[
\sigma(\alpha_n) = \alpha_n + G_\sigma(L)(\beta_n) \quad \text{and} \quad \sigma(\beta) = F_\sigma(L)(\beta) \ , \forall \beta \in \ker(L^n)\ .
\]
Such $G_\sigma, F_\sigma$ make $\iota_n$ a morphism of groups which is well-defined and injective: every $\beta\in \ker(L^n)$ is a difference of roots of $f^n$, and hence it lies in the splitting field of $f^n$. Together with $\alpha_n$ all such $\beta$'s  generate such splitting field.

From now on, for the sake of brevity, we denote
\[
R_n \coloneqq\tfrac{\F_q[y]}{(y^n)}  \ , \quad (R_n)^\times \coloneqq \left(\tfrac{\F_q[y]}{(y^n)} \right)^\times \ .
\]
Notice that if we choose the $\alpha_n$'s consistently, i.e.\ we impose $f(\alpha_{n+1}) = \alpha_n$ similarly to what we do for the $\beta_n$'s, then the maps $\iota_n$ can be packed together to define the map $\iota$ in \eqref{eq:def_iota}:
\begin{equation}\label{eq:pedantic_iota}
    \varprojlim_n \iota_n  \eqqcolon \iota \,\colon\,  
    G_\infty(f,0) \, \longrightarrow \,\varprojlim \left( R_n \rtimes (R_n)^\times \right) = \F_q[\![ y ]\!] \rtimes \F_q[\![ y ]\!] ^\times.
\end{equation}

Then, for the universal $f$, we have the following lemma, which is a  straightforward generalization of \cite[Theorem 5.3]{juul}.

\begin{lemma}\label{gal_gp_structure}
	Let $K$ be an $\F_q$-field and let $\fgen\coloneqq x^q+tx+s\in K(t,s)[x]$. 
    Let $n\ge 1$ and let $\alpha_n$ be a root of $f^n$. Then, denoting $L_t = x^q+tx$, we have that
    $$K(t,s)_n(f,0)=K(t,s)_n(L_t,0) \cdot K(t,s)(\alpha_n).$$ 
    Moreover, the map $\iota_n$ is an isomorphism
	$$G_n(f,0)\cong R_n\rtimes (R_n)^\times.$$
\end{lemma}
\begin{proof}
	To ease the notation, from now on we let $K_n\coloneqq K(t,s)_n(f,0)$ and $F_n\coloneqq K(t,s)_n(L_t,0)$.
	
    The first claim is a direct consequence of \eqref{n-th iterate} and the observation below. Since the map $\iota_n$ is an embedding, in order to complete the proof we only need to show that $\deg (K_n/F_n) = \# R_n = q^n$ and that $\deg(F_n/K(t,s)) = \# (R_n)^\times$. The second equality follows from the theory of Drinfeld modules, see e.g. \cite[Theorem 2.3]{hayes}: we know that $\gal(L_t^n/k(t))\cong (R_n)^\times$ for any field $k$ of characteristic $p$. Hence this holds in particular for $k=K(s)$.
	
	To prove the first equality, we show that $f^n$ is irreducible over $F_n$. Let $v$ be a place of $K(t)_n(L_t,0)$ above the place corresponding to $t$ in $K(t)$. Consider the valuation ring $\mathcal O_v$ in $K_n(L_t,0)$, that is a local ring. Then $f^n\in \mathcal O_v[s][x]$, as it is clear from \eqref{n-th iterate}. Let $\pi$ be a uniformizer at $v$. The ideal generated by $\pi$ in $\mathcal O_v[s]$ is clearly prime; on the other hand the reduction of $f^n$ modulo $\pi$, that is an element of $k_v[s][x]$ where $k_v$ is the residue field at $v$, is Eisenstein at the prime $s$, and is therefore irreducible. It follows that $f^n$ is irreducible in $\O_v[s][x]$, and since $\O_v[s]$ is a UFD and $f^n$ is clearly primitive, by Gauss' Lemma it follows that $f^n$ is irreducible in $F_n[x]$.
\end{proof}

We now consider the subgroup
$$(R_n)_1^\times\coloneqq \left\{1+\sum_{i=1}^{n-1}a_iy^i\colon a_1,\ldots,a_{n-1}\in \F_q\right\}\le (R_n)^\times.$$

\begin{lemma}   \label{lemma:maximal}
    Denote $H \coloneqq (R_n)_1^\times$.
    We naturally have 
    $$
    \left( R_n\rtimes H \right)^{ab} \cong (R_n/yR_n) \times H 
    \cong  \F_q \times H \ ,
    $$
    and the quotient map from $R_n\rtimes H$ to its abelianization $(R_n/yR_n) \times H$ sends $(\lambda,\mu)$ to $(\lambda \pmod y, \mu)$. Moreover,  all maximal subgroups of $R_n\rtimes H$ are normal, with cyclic quotient of order $p$ and they are the preimages of the maximal subgroups of $(R_n/yR_n) \times (H/H^p) \cong \F_q \times (H/H^p)$.
\end{lemma}
\begin{proof}
The first equality holds because the group of commutators in a semidirect product $G \rtimes_\phi H$ with $G,H$ abelian is the subgroup generated by all elements of the form $(g^{-1} \phi_h(g), 1)$ for $g \in G, h\in H$.

In our situation $g = \lambda(y)$ is a polynomial modulo $y^n$ and $h$ is any  polynomial modulo $y^n$ of the form $1 + y\mu(y)$. Hence 
$$g^{-1} \phi_h(g) = (1+y\mu(y)) \cdot \lambda(y)- \lambda(y) = y\mu(y)\lambda(y) \ ,$$
and we conclude immediately that the set of such commutators is exactly $yR_n$. 
Of course $R_n/yR_n \stackrel{\sim}{\to} \F_q$ simply by mapping $\lambda$ to $\lambda(0)$.
This proves the first part of the statement.

Now, $R_n \times H$ is a $p$-group, hence it is nilpotent and by \cite[Theorem 1.26]{Isaacs} its maximal subgroups are normal and have index $p$. In particular maximal subgroups are preimages of maximal subgroup of $A = (R_n \times H)^{ab}$, and actually of $A/A^p$, which is naturally $(R_n/yR_n) \times (H/H^p)$ by the first part of the statement. 
\end{proof}

We keep $K$ to be an $\F_q$-field and $\fgen \coloneqq x^q+tx+s$. Denote $K_\infty \coloneqq K(t,s)_\infty(\fgen,0)$ and let $\theta \in K_\infty$ be a $(q-1)$-th root of $-t$. 
For every $i\in \Z_{\ge 1}$ with $p\nmid i$ and every $j\in \{1,\ldots,r\}$, let $\chi_{i,j}$ be the continuous character $\gal\left( K(t,s)^{\text{sep}}/K(t,s)(\theta) \right) \to \F_p$ that cuts out the extension of $K(t,s)(\theta)$ generated by a root of $x^p-x-\frac{\gamma_j}{t^i}$. Finally, for every $\ell\in \{1,\ldots,r\}$ let $\psi_\ell$ be the continuous character $\gal\left( K(t,s)^{\text{sep}} /K(t,s)(\theta) \right) \to \F_p$ that cuts out the extension of $K(t,s)(\theta)$ generated by a root of $x^p-x-\frac{\gamma_\ell s}{\theta^q}$. For each of the above characters, denote by $\widetilde{\cdot}$ its restriction to $\gal(K_\infty/K(t,s)(\theta))$.

\begin{lemma}\label{theorem:p-extension_theta}
With the notation as above, the map $\iota$ in \eqref{eq:pedantic_iota} gives an isomorphism of pro-$p$-groups
\begin{equation}\label{eq:GenGal=Phi1}
\gal\left( K_\infty /K(t,s)(\theta) \right) \cong \Phi_{\infty,1}  = \varprojlim_n \left( (R_n^q) \times  (R_n^q)_1^\times \right).   
\end{equation}
Moreover, the sequence
$$(\widetilde{\psi}_1,\ldots,\widetilde{\psi}_r,\widetilde{\chi}_{i,j},\ldots\colon j\in \{1,\ldots,r\}, i\in \Z_{\ge 1}, p\nmid i)$$
is a basis of the $\F_p$-space of continuous characters $\gal\left( K_\infty /K(t,s)(\theta) \right) \to \F_p$.

In other words each minimal subfield of $K_\infty /K(t,s)(\theta)$  is generated by an element with minimal polynomial $x^p-x-u$ with $u$ an $\F_p$-combination, unique up to $\F_p^\times$ scalars, of the elements $\frac{\gamma_j}{t^{i}}$ and $\frac{\gamma_\ell  s}{\theta^q}$.
\end{lemma}
\begin{proof}
    It it enough to prove our statement in the case $K = \F_q$. Indeed Lemma \ref{gal_gp_structure} holds for any $\F_q$-field, hence for any other $K$ we have a diamond of fields
    \[
\begin{tikzcd}
  & K_\infty \arrow[dr, -] & \\
K(t,s) \arrow[ur, -, "\Phi_\infty"] &   & (\F_{q})_\infty \arrow[dl, - , "\Phi_\infty"]\\
  & \F_q(t,s) \arrow[ul, -] &
\end{tikzcd}   
\]
and the restriction of automorphisms from $K_\infty$ to $(\F_{q})_\infty$ gives an isomorphism between $\gal(K_{\infty}/K(t,s))$ and $\gal((\F_{q})_\infty/\F_q(t,s))$. This also gives a bijection between the respective subextensions, and in particular $\F_q(\theta)$ corresponds to $K(\theta)$.
    From now on we suppose $K = \F_q$ and we use the notation $K_\infty = (\F_{q})_\infty$ for the sake of brevity. 

    The isomorphism \eqref{eq:GenGal=Phi1} is a consequence of Lemma \ref{gal_gp_structure}. In fact the Galois group of $K_\infty/\F_q(t,s)$ is 
    $\Phi_\infty = \varprojlim  (R_n^q) \times  (R_n^q)^\times$, and we now see that the subextension $\F_q(t,s)(\theta)$ is the fixed field of $\Phi_{\infty, 1}$: following \eqref{eq:beta_i}, we have $\theta = \beta_1 \in \ker(L)$, hence an automorphism $\sigma \leftrightarrow (G_\sigma,F_\sigma)$, with $F_\sigma = a_0 + \ldots + a_d y^d$ acts on $\theta$ as 
    $$ \sigma(\theta) = F_\sigma(L)(\theta) = a_0\theta \ .$$
    
    Every continuous character from $\Phi_{\infty,1}$ to $\F_p$ factors through its abelianization, which by Lemma \ref{lemma:maximal} is  
    $$ \Phi_{\infty,1}^{\text{ab}} = (\llbracket\F_q\rrbracket/(y)) \times R_{\infty,1} =  \F_q \times R_{\infty,1}  \quad \text{with } \quad R_{\infty,1} = (1+y\F_q\llbracket y\rrbracket).
    $$
    In particular any $\F_p$-character of $\phi_{\infty,1}$ can be written uniquely as the product of a character $\eta\colon R_{\infty,1} \to \F_p$ and a character $\xi \colon \F_q \to \F_p$. 
    

    Let $\pi\colon  \Phi_{\infty,1} \to R_{\infty,1}$ be the natural projection.
    The space of characters  $\F_q \to \F_p$ is $r$-dimensional, hence if we find a basis $(\mu_1, \mu_2 \ldots)$ for the characters $ R_{\infty,1} \to \F_p$, and $r$ linearly independent characters $\eta_1,\ldots,\eta_r \colon  \Phi_{\infty,1} \to \F_p$ that factors through $\Phi_{\infty,1} \to  \Phi_{\infty,1}^{\text{ab}}  \to \F_q$, then the $\eta_i$'s, together with  
    the induced characters 
    $ \Phi_{\infty,1} \overset\pi\to R_{\infty,1} \overset {\mu_i}\to \F_p$ for $i\ge1$, form a basis of the space of all continuous characters. 

   Consider the natural projection $ \Phi_{\infty,1} \to  \Phi_{\infty,1}^{\text{ab}}  \to \F_q$, whose fixed field is exactly  the splitting field $M$ of $f$ over $K(t,s)(\theta)$; such splitting field, after a change of variable, is defined by the equation $x^q-x+s/\theta^q=0$. The field extension $M/K(t,s)(\theta)$ has precisely $r$ cyclic $p$-subextensions. 
   In order to describe them, first notice that since $K$ contains $\F_q$, then $M$ contains the splitting field of $(\gamma x)^q-\gamma x+s/\theta^q$ for any $\gamma\in \F_q$, that is equivalent to saying that $M$ contains the splitting field of $x^q-x+\gamma s/\theta^q$ for every $\gamma\in \F_q$. On the other hand one has the following elementary fact: if $L$ is an $\F_q$-field and $u\in L$ takes the form $z^q-z$ for some $z\in L$, then $u\in \mathcal U(K)$. In fact, if say $q=p^k$ and $w\coloneqq \sum_{i=1}^{k-1}z^{p^i}$, then $w^p-w=u$. This shows that $M$ contains the splitting field of $x^p-x+\gamma s/\theta^q$ for every $\gamma\in \F_q$. Now letting, for every $j\in \{1,\ldots,r\}$, $\eta_j$ be the character of $K_\infty$ that cuts out $x^p-x+\gamma_js/\theta^q$, one sees easily that $\eta_1,\ldots,\eta_r$ are $\F_p$-linearly independent, and of course they coincide with $\widetilde{\psi}_1,\ldots,\widetilde{\psi}_r$. Since they are exactly $r$, they form a basis of the space of characters $\Phi_{\infty,1}$ that factor through $\Phi_{\infty,1}^{ab}/R_{\infty,1}\cong \F_q$.
    
    We now look for the $\mu_i \colon R_{\infty,1} \to \F_p$. Notice that, since $(R_n^q)^\times = (R_n^q)_1^\times\times \F_q^\times$ canonically (indeed $(R_n^q)_1^\times$ is the $p$-Sylow of the abelian group $(R_n^q)^\times$), we also  have 
    $$R_\infty^\times \coloneqq \F_q\llbracket y\rrbracket^\times = \varprojlim (R_n^q)^\times \cong \F_q^\times \times R_{\infty,1} \ .$$
    In particular, since $\F_q^\times$ has order coprime to $p$, characters $ R_{\infty,1} \to \F_p$ extend and correspond to characters  $R_{\infty}^\times \to \F_p$. By Lemma \ref{gal_gp_structure} we have
    $$ R_{\infty}^\times = \gal(\F_q(t,s)_\infty(L_t,0)/\F_q(t,s)) = \gal(\F_q(t)_\infty(L_t,0)/\F_q(t)),$$
    so that the kernel of a non-trivial character $\mu\colon R_{\infty}^\times \to \F_p$ 
    corresponds to a cyclic $p$-extensions  $M/\F_q(t)$ contained in $\F_q(t)_\infty(L_t,0)$, (or, equivalently, when we consider $\mu$ as a character on $\Phi_{\infty,1}$), to the  cyclic $p$-extensions  $M \cdot \F_q(t,s)(\theta) /\F_q(t,s)(\theta)$. 

    Now we remark that the operator $L_t$ is nothing else than the operator $\phi_t$ for $\phi$ the Carlitz module defined in Section \ref{sec:preliminaries}; it follows that the field $\F_q(t)_\infty(L_t,0)$ 
    is exactly the union of all the  $\F_q(t)[\Lambda_{t^i}]$'s, 
    which by Theorem \ref{thm:Hayes} is ramified only at the places at $t=0$ and infinity, and it is tamely ramified at infinity, with inertia degree equal to $1$. In particular its $p$-cyclic subextensions are only ramified at $t=0$
    and totally split at $\infty$, and by  Corollary \ref{cor:all_pexts_are_where_we_like} all such $p$-cyclic extensions of $\F_q(t)$ are contained in $\F_q(t)_\infty(L_t,0)$. 
    Corollary \ref{all_cyclic_unramified} now implies immediately the claim.
\end{proof}

\section{Proof of Theorem \ref{thm:main}}

We use a specialization argument.

Since the polynomial $f = f_{t_0,s_0} \in K[x]$ in Theorem \ref{thm:main} is a specialization of $\tilde f  = f_{t,s} \in K(t,s)[x]$ studied in Section \ref{sec:generic} and since the iterates of both polynomials are separable, then we can view the Galois $G_{\infty}(f,0)$ of $K(f,0)_\infty/K$ as a subgroup of the Galois of $G_{\infty}(\tilde f,0)$, which is isomorphic to $\Phi_\infty$ by Lemma \ref{gal_gp_structure}. 

To be more precise we describe the embedding of the two Galois groups as follows. 
Denote $L = L_{t_0} = x^q+t_0x$ and $\tilde L = L_{t} = x^q+tx$.
Let $\calO = K[t,s,t^{-1}] \subset K(t,s)$ and let
$\calO_n = \calO[\tilde \alpha_n,\tilde \beta_n, \theta] \subset K(t,s)_n(f,0)$, with $\theta^{q-1}=-t$, $\tilde \beta_1 = \theta$,  $\tilde\beta_n$ a root of $\tilde L(x) - \tilde \beta_{n-1}$ as in \eqref{eq:beta_i}, and analogously
$\tilde \alpha_1$ a root of $\tilde f(x)$ and $\tilde \alpha_n$ of $\tilde f(x)-\tilde \alpha_{n-1}$. 
Then, choosing analogous roots $\alpha_n, \beta_n$ of $f^n, L^n$ gives maps 
$$ \psi = \psi_n \colon \calO_n \to K_n(f,0) \ , \quad \psi(\tilde \beta_n) = \beta_n\,, \quad \psi(\tilde \alpha_n) = \alpha_n \ .$$
In particular, if we identify $G_{\infty}(\tilde f,0)$ with $\Phi_\infty$ using the $\iota$ map for $\tilde f$ (see Lemma \ref{gal_gp_structure}), then the $\iota$ map relative to $f$ gives an embedding $G_{\infty}(f,0) \hookrightarrow G_{\infty}(\tilde f,0)$. Since we imposed $\psi(\tilde \alpha_n) = \alpha_n$, $\psi(\tilde \beta_n) = \beta_n$, for each $\sigma\in  G_{\infty}(f,0)$ and each $z\in \calO_n$ we have 
$$\psi(\sigma z) =\sigma \psi(z) \ .$$
Notice that by how  $G_{\infty}(\tilde f,0)$ acts in \eqref{eq:Gnf0_emb}, then $\calO_n$ is Galois-stable.



We notice a technical fact about $\psi$ for future reference. If $\alpha \in \calO_n$ is an element with minimal polynomial over $K(t,s)$ equal to  $\mu_\alpha = x^d + a_{d-1}x^{d-1} + \ldots + a_0 \in \calO[x]$ 
and if $\psi(\mu_\alpha) = x^d + \psi(a_{d-1})x^{d-1} + \ldots + \psi(a_0)$ is separable in $K[x]$, then all the Galois conjugates $\alpha_1, \ldots, \alpha_d$ of $\alpha$ have different images $\psi(\alpha_i)$: indeed, writing $\mu_\alpha = \prod (x-\alpha_i)$, we get a factorization $\psi(\mu_\alpha)  = \prod (x-\psi (\alpha_i))$, and the separability of $\psi(\mu_\alpha)$ gives the distinctness of $\psi(\alpha_i)$'s.

We start proving the second part of the statement, namely the case where $\theta = \sqrt[{q-1}]{-t}$ belongs to $K$. In particular the map $\calO \to K$ extends to a map $\calO' = \calO[\theta] \to K$ sending $\tilde \beta_1 = \theta \mapsto \beta_1 = \theta_0$ and, using Lemma \ref{theorem:p-extension_theta}, the map $\iota$ gives an embedding 
$$G_{\infty}(f,0) \hookrightarrow \gal(K(t,s)_\infty(\tilde f,0)/K(t,s)(\theta)) = \Phi_{\infty,1} \ .$$
Indeed if $\sigma \in G_{\infty}(f,0)$ does not lie in $\Phi_{\infty,1}$, it means that $\sigma \theta \neq \theta$, which, since $\psi(\mu_\theta) = x^{q-1}+t_0$ is separable, implies that $\sigma(\theta_0) \neq \theta_0$ which leads to a contradiction.

Then $G_{\infty}(f,0)$ is not the full $\Phi_{\infty, 1}$ if and only if it is contained in a proper maximal subgroup of $\Phi_{\infty, 1}$, hence in a proper maximal closed subgroup, since  
the map $\iota$ is obtained as a limit of the maps in \eqref{eq:Gnf0_emb}, hence it is continuous. 
Such maximal closed subgroups $H$ correspond to minimal subfield of $K(t,s)_\infty(\tilde f, 0)/K(\theta)$, which are characterized in Lemma \ref{theorem:p-extension_theta}: they are generated by the roots of 
$$f_u(x)  = x^p-x-u \quad \text{ for certain }u = u(H)\ .$$ 
In particular, $G_{\infty}(f,0)$ is contained in a conjugate of $H$ if and only if $x^p-x- \psi(u(H))$ has a root in $K$: 
given $\alpha$  a root of $f_u$, then $G_{\infty}(f,0)$ is contained in a conjugate of $H$ if and only if there exists a conjugate $\alpha'$ of $\alpha$ such that for all $\sigma \in G_{\infty}(f,0)$ we have $\sigma \alpha' = \alpha'$; anyway, since $\psi(\mu_\alpha) = \psi(\mu_{\alpha'}) = x^p-x- \psi(u(H))$ is separable, the images under $\psi$ of all conjugates of $\alpha$ are distinct and are exactly the roots of $\psi(\mu_\alpha)$, hence ``$\sigma \alpha' = \alpha'$ for all  $\sigma \in G_{\infty}(f,0)$'' is equivalent to ``$\sigma \psi(\alpha') = \psi(\alpha')$ for all  $\sigma \in G_{\infty}(f,0)$'', which is equivalent to $\psi(\mu_\alpha)$ having a root in $K$.

We deduce that $G_{\infty}(f,0)$ is not contained in a maximal closed subgroup if and only $K$ if, for all $u \neq 0$ as in Lemma \ref{theorem:p-extension_theta}, the polynomial  $x^p-x- \psi(u)$ does not have a root in $K$. By Artin-Schreier theory, this is equivalent to saying that the $\psi(u_i)$'s are independent in $\calV(K)$, for $(u_i)$ a basis of the possible $u\in \calV(K(\theta))$. The basis given in Lemma \ref{theorem:p-extension_theta} corresponds, under $\psi$, to the basis given in the statement of Theorem \ref{thm:main}.

We now turn to the first part of the Theorem \ref{thm:main}.
Assuming $q>2$, we want to characterize when the embedding $\iota \colon G_{\infty}(f,0) \to G_{\infty}(\tilde f,0) = \Phi_{\infty}$ is surjective. A necessary condition for this to happen is that the composition 
\begin{equation} \label{eq:modp-1}    
G_{\infty}(f,0) \to G_{\infty}(\tilde f,0) \to \gal(K(t,s)(\theta)/K(t,s)) = \F_q^\times  \ ,
\end{equation}
is surjective. The above map only depends on the Galois action on $\theta$ and $\theta_0$, i.e. it factors through the simpler map 
$\gal(K(\theta_0)/K )\to \gal(K(t,s)(\theta)/K(t,s))$, which is surjective if and only if  $K(\theta_0)/K$ has degree $q{-}1$, which by Kummer theory is equivalent to $-t_0$ not being an $\ell$-th power for all $\ell\mid (q-1)$, which is the first condition stated in the theorem. 

Moreover $\iota$ is surjective if and only if the map \eqref{eq:modp-1} is surjective and its kernel $\Psi_{\infty, 1}$ is equal to $G_{\infty}(f,0) \cap \Psi_{\infty, 1}$. The latter is the Galois group of $K_\infty(f,0)$ over $K(\theta_0)$, i.e the dynamical Galois group of $f$ considered as a polynomial over $K(\theta_0)$, hence, by the second part of the theorem, which we have already proven, $\Psi_{\infty, 1}=G_{\infty}(f,0) \cap \Psi_{\infty, 1}$ if and only if the set
$$\left\{\frac{	\gamma_js_0}{\theta_0^q}\colon j\in \{1,\ldots,r\}\right\}\cup\left\{\frac{\gamma_j}{t_0^i}\colon j\in \{1,\ldots,r\},\,i\ge 1,\,p\nmid i\right\}$$
is $\F_p$-linearly independent in $\calV(K(\theta_0))$.

We are left to prove, assuming that $K(\theta_0)/K$ has degree $q{-}1$, that this condition is equivalent to conditions $(2)$ and $(3)$ in the theorem.
One implication is clearer: if the $\gamma_j/t^i$ are independent in $V(K(\theta_0))$ then they are also independent in $\calV(K)$. Moreover, if $\{\gamma_1s_0/\theta_0^q,\ldots,\gamma_rs_0/\theta_0^q\}$ is in an independent set in $\calV(K(\theta_0))$, then the Galois group of the splitting field $M$ over $K(\theta_0)$ of the polynomial $x^q-x-s_0/\theta_0^q$ has $r$ independent characters to $\F_p$; since such Galois group is contained in $\F_q$, it must then coincide with it. But of course up to a change of variables this is the same as the Galois group of $f$. By the degree rule, this shows that the Galois group of $f$ over $K$ is indeed $\F_q\rtimes\F_q^\times$. In particular, by \eqref{eq:Gnf0_emb} the Galois group of $f$ permutes all its roots, and hence $f$ is irreducible.

Now the other implication. Assume that $(1)$, $(2)$ and $(3)$ hold, so that in particular the extension $K(\theta_0)/K$ has degree $q{-}1$. First of all if the $\gamma_j/t_0^i$'s are independent in $\calV(K)$, then they are also independent in $\calV(K(\theta_0))$: indeed independence of the $\gamma_j/t_0^i$ in $\calV(K)$ is equivalent  to the corresponding Artin-Schreier extensions being linearly disjoint over $K$, which, since these extensions are Galois of degree $p$ and since $K(\theta_0)/K$ has degree coprime to $p$, is equivalent to the 
corresponding Artin-Schreier extensions being linearly disjoint over $K(\theta_0)$, which is what we claim. 
We are left to prove that $\{\gamma_1s_0/\theta_0^q,\ldots,\gamma_rs_0/\theta_0^q\}$ is linearly independent from the $\gamma_j/t_0^i$ in $\calV(K(\theta_0))$. 
First of all the irreducibility of $f$ implies that the cardinality of $\gal(f/K)$ is a multiple of $q$ and, since it is also multiple of $[K(\theta_0):K] = q-1$, that the embedding 
$    \iota_1 \colon \gal(f/K) \to  \F_q \rtimes \F_q^\times  $   
is an isomorphism. 
 In particular, we also have 
\begin{equation}\label{eq:last_iota_11_surg}
\gal(K(f,0)/K(\theta_0))\cong \gal(K(\tilde f,0)/K(\theta)) \cong \F_q \ ,
\end{equation}
and by Theorem \ref{theorem:p-extension_theta} the cyclic subextensions of $K(f,0)/K(\theta_0)$ correspond to $\F_p$-combinations of $\gamma_js_0/\theta_0^q$, up to $\F_p^\times$.

Now suppose by contradiction that there is a non-trivial relation
\begin{equation}\label{eq:last_linear_relation}
\sum _{i,j} b_{i,j} \frac{\gamma_j}{t_0^i} = a_1 \frac{\gamma_1 s_0}{\theta_0^q} + \ldots +a_r \frac{\gamma_rs_0}{\theta_0^q}  \quad \text{in } \calV(K(\theta))
\ ,
\end{equation}
for suitable $a_\ell,b_{i,j} \in \F_p$, with $b_{i,j}=0$ except for finitely many indices.
In particular, if $\gamma \coloneqq a_1 \gamma_1 + \ldots + a_r \gamma_r \in \F_q$ and $U \coloneqq \sum _{i,j} b_{i,j} \gamma_j T^i \in \F_q[T]$, then the above relation implies that the polynomials
$x^p-x-\gamma \frac{s_0}{\theta_0^q}$ and $x^p-x- U(1/t_0)$ define the same extension $M$ of $K(\theta_0)$.
Such extension needs to be abelian over $K$: the polynomial $x^p-x- U(1/t_0)$ defines an abelian extension of $M'/K$, and $M/K$ is the compositum of the two abelian extensions $K(\theta_0)/K$ and $M'/K$. 

We now derive a contradiction looking at the polynomial $x^p-x-\gamma \frac{s_0}{\theta_0^q}$, which lets us interpret $M$ as a subfield of $K(f,0)$. 
Since \eqref{eq:last_linear_relation} is a nontrivial relation and the $\gamma_j/t_0^i$'s are $\F_p$-independent, the vector $(a_1, \ldots, a_r)$ is nonzero. Then, by \eqref{eq:last_iota_11_surg}, the polynomial $x^p-x-\gamma \frac{s_0}{\theta_0^q}$ defines a $p$-cyclic subextension of $K(f,0)/K(\theta_0)$, namely the fixed field of a certain subgroup $H \subseteq \F_q$ of index $p$.
If $q=p$, the only index $p$ subgroup is $H = \{ 0\} \subseteq \F_p$, implying that $M = K(f,0)$. It follows that $\gal(M/K) = \F_q \rtimes \F_q^\times $, that is not abelian for $q>2$.
If $q>p$, the subgroup $H \subseteq \F_q$ must be proper.
However, the subgroup $H\rtimes \F_q^\times\subseteq \F_q \rtimes \F_q^\times = \gal(K(f,0)/K)$ is not normal, since the action of $\F_q^\times $ on $\F_q$ is transitive on $\F_q\setminus \{0\}$. This contradicts the fact that $M/K$ is Galois, and concludes therefore the proof.

\bibliographystyle{plain}
\bibliography{bibliography}
\end{document}